\documentclass[reqno, 12pt, reqno, twoside]{amsart}

\usepackage{amsfonts, amsthm, amsmath, amssymb, bbm, bm, enumerate}

\usepackage{hyperref}
\usepackage[all]{xy} 
\usepackage[margin=1.25in]{geometry}

\usepackage{helvet}

\RequirePackage{mathrsfs} \let\mathcal\mathscr

\numberwithin{equation}{section}

\newtheorem{theorem}{Theorem}[section]

\newtheorem{lemma}[theorem]{Lemma}

\theoremstyle{definition}

\renewcommand{\phi}{\varphi}
\renewcommand{\rho}{\varrho}

\renewcommand{\AA}{\mathbb{A}}

\newcommand{\ZZ}{\mathbb{Z}}

\newcommand{\QQ}{\mathbb{Q}}
\newcommand{\RR}{\mathbb{R}}

\renewcommand{\leq}{\leqslant}
\renewcommand{\geq}{\geqslant}

\newcommand{\x}{\mathbf{x}}

\DeclareMathOperator{\Mod}{mod}

\newcommand{\Br}{{\rm Br}}

\newcommand{\sumstar}{\sideset{}{^*}\sum}
\newcommand{\modstar}[1]{\left(\ZZ/#1 \ZZ\right)^{\times}}

\title[Hasse principle failures in a family of surfaces]{A positive proportion of Hasse principle failures in a family of Ch{\^a}telet Surfaces}
\author{Nick Rome}

\begin{document}

\address{School of Mathematics\\
University of Bristol \\ Bristol\\ BS8 1TW\\ UK}
\email{nick.rome@bristol.ac.uk}

\begin{abstract}
We investigate the family of surfaces defined by the affine equation $$Y^2 + Z^2 = (aT^2 + b)(cT^2 +d)$$ where $\vert ad-bc \vert=1$ and develop an asymptotic formula for Hasse principle failures. We show that a positive proportion (roughly 23.7\%) of such surfaces fail the Hasse principle, by building on previous work of la Bret\`{e}che and Browning.\\
\end{abstract}

\date{\today}

\maketitle

\thispagestyle{empty}
\section{Introduction}\label{intro}

A Ch{\^a}telet surface over a number field $K$ is a smooth proper model of the affine surface defined by the equation
$$Y^2 - e Z^2 = f(T),$$
where $e \in K^*$ is non-square and $f \in K[T]$ is a separable polynomial of degree 3 or 4. 
The work of Colliot-Th{\'e}l{\`e}ne, Sansuc and Swinnerton-Dyer \cite{CTSSD} shows that the Hasse principle holds for these surfaces unless $f$ is a product of two irreducible quadratic polynomials over $K$, in which case there may be counterexamples. In this paper, we restrict attention to the case $K = \QQ$ and the Ch{\^a}telet surfaces $X_{a,b,c,d}$, defined by the equation
\begin{equation}\label{def}Y^2 + Z^2 = (aT^2 +b)(cT^2 +d), \end{equation}
for $(a,b,c,d) \in \QQ^4$ such that $abcd \neq 0$ and $ad-bc \neq 0$.

There exist several explicit counterexamples to the Hasse principle for surfaces of this type in the literature. Colliot-Th{\'e}l{\`e}ne, Coray and Sansuc \cite{CTCS} provide an infinite family of counterexamples over $\mathbb{Q}$ given by $X_{1,1-k,-1,k}$ for any positive integer $k \equiv 3$ mod 4. This generalises the earlier example of Iskovskikh \cite{Isk} where $k=3$.
An investigation of the surfaces defined by \eqref{def} was undertaken by la Bret{\`e}che and Browning \cite{BandB}, in which they vary the 4-tuple of coefficients $(a,b,c,d)$ and develop asymptotics for what proportion produce counterexamples to the Hasse principle. The major results of \cite{BandB} showed that roughly $83.3\%$ of Ch{\^a}telet surfaces have local solutions everywhere but $0\%$ are counterexamples to the Hasse principle. We aim to discover a large family of surfaces among which a positive proportion fail the Hasse principle.

When conducting investigations of the Hasse principle it is common to phrase questions in terms of the Brauer group.  In \cite{CTSSD}, Colliot-Th{\'e}l{\`e}ne et al showed that the only obstruction to the Hasse principle for these surfaces is explained by the Brauer--Manin obstruction. In the case of surfaces $X=X_{a,b,c,d}$ of the form \eqref{def}, we have $\Br X / \Br \QQ \cong\ZZ/2\ZZ$ and that the quotient is generated by the quaternion algebra $(-1, aT^2 +b)$. Hence the Brauer--Manin obstruction can be made very explicit. Given $(a,b,c,d) \in \QQ^4$ we have $$X(\AA_{\QQ})^{\Br}=\left\{(x_{\nu},y_{\nu},t_{\nu})_{\nu \in \Omega} \in X(\AA_{\QQ}):\prod_{\nu \in \Omega}(-1,at_{\nu}^2+b)_{\nu} = +1\right\},$$ where $\Omega$ denotes the set of places of $\QQ$ and $(\cdot,\cdot)_{\nu}$ denotes the Hilbert symbol associated to the field $\QQ_{\nu}$. Since $X(\QQ) \subset X(\AA_{\QQ})^{\Br}$, to produce counterexamples to the Hasse principle we simply need to find choices of coefficients $(a,b,c,d)$ such that the above product of Hilbert symbols is made $-1$ for any point $(x_{\nu},y_{\nu},t_{\nu})_{\nu \in \Omega} \in X(\AA_{\QQ})$. We will see in Section 3 that a way to do this is to impose the condition $\vert ad-bc\vert=1$ which will force the Hilbert symbol $(-1, aT^2 + b)_p$ to be constant at all odd primes $p$. Note that this condition is satisfied by surfaces in the Colliot-Th{\'e}l{\`e}ne--Coray--Sansuc family mentioned above.


To count the surfaces in this family, we make some initial reduction steps. We will want to identify $(a,b,c,d)$ with $(-a,-b,-c,-d)$ as this does not affect the surface in any way. Clearly switching the order of the two brackets on the right hand side of equation \eqref{def} will not change the equation either therefore we see that $X_{a,b,c,d}$ is invariant under the map
$$\rho_1: (a,b,c,d) \mapsto (c,d,a,b).$$
Moreover, since the surface is a projective variety, it makes no difference if we divide by $T$ in \eqref{def} and relabel $Y$ and $Z$ accordingly. Therefore $X_{a,b,c,d}$ is also invariant under $$\rho_2: (a,b,c,d) \mapsto (b,a,d,c).$$ For the same reason, $X_{a,b,c,d}$ is equivalent to $X_{\lambda a, \lambda b, \lambda c, \lambda d}$ for any $\lambda \in \QQ^*$.
Each Ch{\^a}telet surface can be written as $X_{a,b,c,d}$ where $(a,b,c,d) \in \ZZ^4_{\text{prim}}$. For such a tuple $(a,b,c,d)$ the expression $\vert ad-bc \vert =1$ is invariant under $\varrho_1$ and $\varrho_2$.
We introduce the following set which will act as our parameter space for Ch{\^a}telet surfaces $$\mathscr{M}:=\{(a,b,c,d) \in \ZZ^4/\{\pm1\}: \vert ad-bc \vert=1, \, abcd \neq 0\}/\sim,$$ where two tuples are equivalent under $\sim$ if they are in the same orbit of $\rho_1$ and $\rho_2$. We will be interested in the surfaces which have local solutions everywhere and those which have local solutions everywhere but no global solutions. These sets will be denoted $\mathscr{M}_{\text{loc}}$ and $\mathscr{M}_{\text{Br}}$, respectively. 
On choosing the sup-norm for points ${u=(a,b,c,d)}$ on $\mathscr{M}$, i.e. $|u|=\max\{|a|,|b|,|c|,|d|\}$ , we can define the counting functions
\begin{equation}\label{locbr}N_{\text{loc/Br}}(P) = \# \{u \in \mathscr{M}_{\text{loc/Br}}: |u| \leq P \}.\end{equation}
The object of this paper is to establish the following asymptotic formulae for these counting functions.

\begin{theorem}\label{theorem}
There exist $\theta, \theta' > 0$ such that
\begin{equation}\label{N_loc}
N_{\text{loc}}(P) = \frac{279}{16\pi^2} P^2 + O\left(P^{2 - \theta}\right) \end{equation}
and 
\begin{equation}\label{N_Br}
N_{\text{Br}}(P) = \frac{33}{8\pi^2} P^2 + O\left(P^{2 - \theta'}\right).
\end{equation}
\end{theorem}

Denoting by $N(P)$ the total number of surfaces in this family, it is established in Theorem \ref{Stot} that  $$N(P) \sim \frac{32}{\pi^2}P^2.$$ Theorem \ref{theorem} therefore shows that a positive proportion of surfaces of this type fail the Hasse principle.

\subsection*{Layout of the paper}
In Section 2, we describe a set of representatives for the set we are interested in counting. Section 3 is devoted to investigating the Brauer group to determine local conditions on the coefficients $(a,b,c,d)$ that ensure local or global solubility. Finally in Section 4, the number of points satisfying the desired local conditions are counted to complete the proof of Theorem \ref{theorem}.
\subsection*{Acknowledgements}
The author is grateful to Tim Browning for suggesting the problem and his continued guidance. Thanks are also extended to Martin Bright and Julian Lyczak for their insight into the Brauer--Manin obstruction.

\section{Set-up}

We will choose as a set of representatives for $\mathscr{M}$ the following set,
\begin{equation}\label{tot}S_{\text{tot}} =\{(a,b,c,d) \in \mathbb{Z}^4 :a>0,  \vert ad-bc \vert =1, abcd \neq 0 \}.\end{equation}
Observe that each tuple in  $\mathscr{M}$ is counted by 4 distinct elements of $S_{\text{tot}}$ due to the invariance under $\varrho_1$ and $\varrho_2$. This means there is a 1:4 correspondence between $\mathscr{M}$ and $S_{\text{tot}}$. It will be important later on to keep track of the 2-adic valuation of the coefficients. In order to simplify this, we will restrict to representatives for which $a$ is odd. We may do this by breaking up $S_{\text{tot}}$ as $$S_{\text{tot}}^{(1)} \sqcup S_{\text{tot}}^{(2)}$$ where 
$S_{\text{tot}}^{(1)} := \{ (a,b,c,d) \in S_{\text{tot}} : 2 \nmid a\}$ and $S_{\text{tot}}^{(2)} := \{ (a,b,c,d) \in S_{\text{tot}} : 2 \mid a\}$. Due to the expression $\vert ad-bc \vert= 1$ we know that when 2 divides $a$, it cannot divide $b$. Therefore by the invariance of Ch{\^a}telet surfaces under the action of $\varrho_2$ we may write $S_{\text{tot}}^{(2)} := \{ (a,b,c,d) \in S_{\text{tot}} : 2 \nmid a, 2 \mid b\}$.

%
In this paper, we'll be interested in those elements of $S_{\text{tot}}$ that define surfaces $X = X_{a,b,c,d}$ with points everywhere locally, especially those without points globally. To this end, we observe that the counting functions defined in \eqref{locbr} may be written as
\begin{align*}
N_{\text{loc}}(P) &= \frac{1}{4}\# \{(a,b,c,d) \in S_{\text{tot}}: |(a,b,c,d)| \leq P, X(\QQ_{\nu}) \neq \emptyset \forall \nu \in \Omega\}\\
N_{\text{Br}}(P) &= \frac{1}{4}\# \{(a,b,c,d) \in S_{\text{tot}}: |(a,b,c,d)| \leq P, X(\QQ_{\nu}) \neq \emptyset \forall \nu \in \Omega, X(\QQ) = \emptyset\}.
\end{align*}

\section{Brauer group considerations}
In this section we aim to work with elements of the Brauer group of the surfaces in our family to show that the Brauer--Manin obstruction is completely controlled by what happens at the prime 2. Before doing so we note that it is possible to obtain the same results of this paper by instead following the work of la Bret{\`e}che and Browning more closely. In \cite{BandB}, they develop explicit conditions on the parameters $(a,b,c,d)$ that guarantee local solubility of the surface $X_{a,b,c,d}$. By investigating a related torsor they are then able to construct similar conditions controlling global solubility. Summing over the $(a,b,c,d)$ that satisfy the local conditions but not the global conditions allows them to develop their asymptotics. One could specialise their results to the case when $\vert ad-bc\vert=1$ however we choose instead to work directly with the Brauer group, for the sake of brevity and to keep the present paper as self contained as possible.

\begin{lemma}
Let $(a,b,c,d) \in S_{\text{tot}}$.
If $p$ is an odd prime then we have $$(-1, at_p^2 + b)_p = +1,$$ for any  ${(x_p, y_p, t_p) \in X_{a,b,c,d}(\QQ_p)}$.
\end{lemma}

\begin{proof}
This is immediate for $p \equiv 1 \Mod 4$. If $p \equiv 3 \Mod 4$ then $$(-1, at_p^2 +b)_p = (-1)^{v_p(at_p^2+b)}.$$
Note that since $$x_p^2+y_p^2 = (at_p^2+b)(ct_p^2+d),$$ we have $$(-1)^{v_p(at_p^2+b)} = (-1)^{v_p(ct_p^2+d)}.$$
Now observe that 
\begin{align*}
a(ct_p^2+d) -c(at_p^2 +b)  = ad - bc = \pm 1.
\end{align*}
Hence either $v_p(at_p^2 +b) = 0$ or $v_p(ct_p^2+d)=0$.
\end{proof}

We now turn to investigating the real place. Recall that $a>0$ for $(a,b,c,d) \in S_{\text{tot}}$. There are four possible combinations of signs that the coefficients $a,b,c$ and $d$ can take while still maintaining the relationship $ad-bc=\pm1$. In particular,
$$\sigma(b,c,d) \in \left\{
		(+,+,+),
		(+,-,-)  ,
		(-,+,-),
		(-,-,+)
\right\},$$ where $\sigma$ is the sign function.
The third case is the subject of Lemma 4.3 in \cite{BandB} wherein it is shown that surfaces defined by coefficients with this signature always satisfy the Hasse principle. In all other cases, the Hilbert symbol is constant. 

\begin{lemma}\label{realhilb}
Let $(a,b,c,d) \in S_{\text{tot}}$ such that $\sigma(b,c,d) \neq (-,+,-)$. Then the Hilbert symbol $(-1,at_{\infty}^2+b)_{\infty}$ is constant as a function of $(x_{\infty},y_{\infty},t_{\infty}) \in X_{a,b,c,d}(\RR).$
\end{lemma}

\begin{proof}
The definition of the real Hilbert symbol tells us that $$(-1,at_{\infty}^2 +b)_{\infty} =+1 \iff at_{\infty}^2 +b >0.$$
Since $a>0$, if $b>0$ we immediately have that $(-1,at_{\infty}^2+b)_{\infty} =+1.$
It remains to investigate the last case, $\sigma(b,c,d) = (-,-,+)$. In this case, we claim that $$(-1,at_{\infty}^2+b)_{\infty} = ad-bc.$$

 Let $(x_{\infty},y_{\infty},t_{\infty}) \in X_{a,b,c,d}(\RR)$ and suppose that  $(-1,at_{\infty}^2+b)_{\infty}=-1$. Since
$$x_{\infty}^2 + y_{\infty}^2 = (at_{\infty}^2 - |b|)(d - |c|t_{\infty}^2),$$ we have $at_{\infty}^2 - |b|<0$ and $d- |c|t_{\infty}^2<0$. Hence $$\frac{d}{|c|}< t_{\infty}^2 < \frac{|b|}{a}$$ and thus $$ad<(a|c|)t_{\infty}^2 < |bc|.$$ 
This inequality is soluble only if $ad-bc=-1$. Similarly, 
if $(-1, at^2_{\infty}+b)_{\infty} = +1$ then we have $at_{\infty}^2 + b >0$ and thus $ct_{\infty}^2 + d>0$. We deduce
\[
|bc|<(a|c|)t_{\infty}^2 <ad ,
\] which is only soluble when $ad-bc = 1$.

\end{proof}
These two lemmas together tell us that the only source of the Brauer--Manin obstruction will come from the 2-adic properties of the surface $X_{a,b,c,d}$. This was spelled out by la Br{\'e}teche and Browning \cite[Lemmas 4.9-4.15]{BandB} and amounts to ensuring that $(a,b,c,d)$ lie in certain specified congruence classes mod 16. 
Suppose that 
$$\beta = v_2(b), \, \gamma = v_2(c) \, \text{ and } \delta = v_2(d),$$ are the 2-adic valuations of the coefficients and $\sigma(b,c,d) = (\epsilon_2, \epsilon_3, \epsilon_4).$ We make the change of variables
\begin{equation}\label{variables} a = a', \, b = \epsilon_2 2^{\beta}b',\, c = \epsilon_3 2^{\gamma}c',\, d = \epsilon_4 2^{\delta}d'.\end{equation}
We will denote by $H^{\pm}_{\beta, \gamma, \delta}$  the union of congruence classes such that for $ad-bc= \pm1$ we have $$X_{a,b,c,d}(\mathbb{Q}_2) \neq \emptyset \iff (a',b',c',d') \in H_{\beta, \gamma, \delta}^{\pm} \text{ mod } 16,$$ and similarly $\widetilde{H}^{\pm}_{\beta, \gamma, \delta}$ the union of congruence classes such that$$X_{a,b,c,d}(\mathbb{Q}_2) \neq \emptyset \text{ but }X_{a,b,c,d}(\QQ) = \emptyset \iff (a',b',c',d') \in \widetilde{H}^{\pm}_{\beta, \gamma, \delta} \text{ mod }16.$$
Our asymptotic formulae will be produced by counting points lying in residue classes in $H^{\pm}_{\beta, \gamma, \delta}$ or $\widetilde{H}_{\beta, \gamma, \delta}^{\pm}$, respectively, mod 16.

\section{Calculating the asymptotics}
We turn our attention to estimating the counting functions $N(P), N_{\text{loc}}(P)$ and $N_{\Br}(P)$. We have seen in the preceeding sections that in order to do this we must count the number of points in $S_{\text{tot}}$ that lie in particular congruence classes $\Mod 16$. To achieve this we appeal to a result based on dynamical systems (Lemma \ref{counting}), which allows us to count points on the quadric $ad-bc=\pm1$  satisfying a congruence condition. We start by computing $N(P)$, the method for $N_{\text{loc}}$ and $N_{\text{Br}}$ being similar.

\begin{theorem}\label{Stot}
Let $N(P) := \frac{1}{4}\#\{ (a,b,c,d) \in S_{\text{tot}} : |(a,b,c,d)| \leq P\}.$ Then there exists $\theta>0$ such that $$N(P) = \frac{32}{\pi^2}P^2 +O(P^{2 - \theta}).$$
\end{theorem}

To count this, we first denote the total set of allowable signs by $S$, the possible values of $(\beta, \gamma, \delta)$ for elements of $S_{\text{tot}}^{(i)}$ by $L^{(i)}$ and the possible congruence classes elements can lie in by $T_{\beta, \gamma, \delta}^{\pm}$. Recall from Section 2 that there is a 1:4 correspondence between $\mathscr{M}$ and $S_{\text{tot}}$. Explicitly, we denote
\begin{align*}
S &= \{(+,+,+,+), (+,-,+,-),(+,-,-,+), (+,+,-,-)\},\\
L^{(i)} &= \{(\beta, \gamma, \delta) \in \ZZ_{\geq 0}^3: \min\{\beta+\gamma,\delta\} = 0 < \max\{\beta+ \gamma, \delta\}, \beta \geq i-1\},\\
T_{\beta, \gamma, \delta}^{\pm} &= \{ \bm{\xi} \in \left(\modstar{16}\right)^4:  \epsilon_42^{\delta}\xi_1\xi_4 - \epsilon_2 \epsilon_3 2^{\beta + \gamma}\xi_2 \xi_3  \equiv \pm1 \Mod 16 \}.
\end{align*}
Now we make the change of variables described in \eqref{variables} so that our counting problem becomes
\begin{equation}\label{total} N(P) = \frac{1}{4}\sum_{\bm{\epsilon} \in S} \sum_{i \in \{1,2\}} \sum_{(\beta,\gamma, \delta) \in L^{(i)}} \sum_{\bm{\xi} \in T_{\beta, \gamma, \delta}^{\pm}}\left(K^+ + K^-\right),\end{equation}
where $K^{\pm}$ denotes the total number of $(a, b', c', d') \in \mathbb{N}^4$ satisfying:
\begin{enumerate}[(i)]
\item$|(a, 2^{\beta}b', 2^{\gamma}c', 2^{\delta}d')| \leq P$,
\item $\epsilon_42^{\delta}ad'-\epsilon_2\epsilon_32^{\beta + \gamma}b'c' = \pm1$,
\item $(a,b',c',d') \equiv\bm{\xi} \mbox{ mod } 16.$
\end{enumerate}
The following result of Browning and Gorodnik \cite{BandG} provides an asymptotic formula for $K^{\pm}$, the main term of which factors as a product of local densities. The values of these densities are computed in Lemmas \ref{pmu},\ref{infmu} and \ref{2mu}.

\begin{lemma}\label{counting}Let
$F(\x) := \epsilon_42^{\delta}x_1x_2-\epsilon_2\epsilon_32^{\beta + \gamma}x_3x_4 ,$ then there exists $\theta >0$ such that 
$$K^{\pm} =  \mu_{\infty}^{\pm}(P) \prod_p \mu_p^{\pm}(\bm{\xi}) + O\left(\mu_{\infty}^{\pm}(P)^{1-\frac{\theta}{2}}\right),$$
where 
\begin{align*}\mu_p^{\pm}(\bm{\xi}) &= \lim_{t \rightarrow \infty} p^{-3t}\#\{\x \in (\mathbb{Z}/p^t \mathbb{Z})^4 : F(\x)  \equiv \pm1 \Mod p^t , \, \x \equiv \bm{\xi} \Mod p^{v_p(16)} \}\\
\mu_{\infty}^{\pm}(P) &=  \lim_{\eta \rightarrow 0} \frac{1}{\eta} \int \limits_{\substack{0<x_1 \leq P\\ 0 < \epsilon_22^{\beta}x_2 \leq P\\ 0 < \epsilon_32^{\gamma}x_3 \leq P\\ 0 < \epsilon_42^{\delta}x_4 \leq P \\ |F(\x)  \mp 1| < \frac{\eta}{2}}} \mathrm{d}\x.\end{align*}
\end{lemma}

\begin{proof}
This is \cite[Proposition 3.1]{BandG} applied to the variety in which we are interested.
\end{proof}

We now proceed to compute these local density factors. The most important consequences of the following results are that $\mu_p^{\pm}$ is independent of $\bm{\xi}$ and that ${\mu_{\nu}^+ = \mu_{\nu}^-}$ for all $\nu \in \Omega$.
\begin{lemma}\label{pmu}
For $p$ odd and any $\bm{\xi}\in \modstar{16}$, we have $\mu_p^{\pm}(\bm{\xi}) = 1 - \frac{1}{p^2}.$
\end{lemma}
\begin{proof}
For $p>2$, under an obvious change of variables, the density is 
 $$\lim_{t \rightarrow \infty} p^{-3t}\#\{\x \in (\mathbb{Z}/p^t \mathbb{Z})^4 : x_1x_2-x_3x_4  \equiv \pm1 \mbox{ mod } p^t\}.$$

Let $N(p^t)$ be the cardinality in which we are interested. Then we can express this congruence counting problem as an exponential sum.
\begin{align*}
N(p^t) &= p^{-t} \sum_{r, x_1, \dots, x_4 \Mod p^t} e\!\left( \frac{(x_1x_2 - x_3x_4 \mp 1)r}{p^t}\right)\\
&= p^{-t} \sum_{r \Mod p^t} e\!\left(\frac{\mp r}{p^t}\right) \left| \sum_{x_1, x_2 \Mod p^t} e\!\left(\frac{rx_1 x_2}{p^t}\right) \right|^2.
\end{align*}
The sum over $x_1,x_2$ is equal to $p^t (r,p^t)$.
Therefore,
\begin{align*}
N(p^t) &= p^t \sum_{r \text{ mod }p^t} e\!\left(\frac{ \mp r}{p^t}\right) (r, p^t)^2\\
&= p^t \sum_{0 \leq \alpha \leq t} p^{2 \alpha} \sum \limits_{\substack{r \text{ mod }p^t \\ (r, p^t) = p^{\alpha}}} e\!\left(\frac{\mp r}{p^t} \right)\\
&= p^t \sum_{0 \leq \alpha \leq t-1} p^{2 \alpha} \sumstar_{r \Mod p^{t- \alpha}}  e\!\left(\frac{\mp r}{p^{t-\alpha}} \right) + p^{3t}.
\end{align*}
This inner sum is now in the form of the Ramanujan sum $c_{p^{t-\alpha}}(\mp 1)$ so can be explicitly evaluated to give us
$$N(p^t) =p^{3t} -p^t \cdot p^{2(t-1)} = p^{3t}\left(1-p^{-2}\right),$$ from which the result follows.
\end{proof}

\begin{lemma}\label{infmu}
We have
$$\mu_{\infty}^{\pm}(P) = \frac{2P^2 }{2^{\beta + \gamma + \delta}} + O(\log^2P).$$
\end{lemma}

\begin{proof}
Under the change of variables
$$y_1 = x_1/P, \, y_2 = \epsilon_2 2^{\beta} x_2/P, \, y_3 = \epsilon_32^{\gamma} x_3/P, \, y_4 = \epsilon_42^{\delta}x_4/P \text{ and } \lambda = \eta/P^2,$$
we have
$$\mu_{\infty}^{\pm}(P) = \frac{\epsilon_2\epsilon_3\epsilon_4P^2}{2^{\beta+\gamma + \delta}} \lim_{\lambda \rightarrow 0} \frac{1}{\lambda} \int \limits_{\substack{0 < x_i \leq 1\\|x_1x_4 - x_2x_3 \mp 1/P^2| < \lambda /2}}\mathrm{d}\x.$$
Recall the set of allowable signs 
$$S= \{(+,+,+,+), (+,-,+,-),(+,-,-,+), (+,+,-,-)\},$$
and observe that $\epsilon_2\epsilon_3\epsilon_4=1$. We now change to hyperbolic co-ordinates
$$\begin{array}{ll}
u_1 &= \sqrt{x_1x_4}, \,\, v_1 = \log\left(\sqrt{\frac{x_1}{x_4}}\right),\\
u_2 &= \sqrt{x_2x_3}, \, \, v_2 = \log\left(\sqrt{\frac{x_2}{x_3}}\right).
\end{array}$$
The integral above transforms to $16I$ where
\begin{align*}
I :&=\frac{1}{4}\int \limits_{\substack{0 < u_i \leq 1 \\ |u_1^2 - u_2^2 \mp 1/P^2| < \lambda/2}}u_1u_2 \int \limits_{-\log u_i < v_i < \log u_i}\mathrm{d}v_1 \mathrm{d}v_2\mathrm{d}u_1 \mathrm{d}u_2\\ &=\int \limits_{\substack{0 < u_i \leq 1 \\ |u_1^2 - u_2^2 \mp 1/P^2| < \lambda/2}}(u_1\log u_1)( u_2 \log u_2) \mathrm{d}u_1 \mathrm{d}u_2.
\end{align*}

At this point, we observe that one can change the sign of the $\pm$ symbol in the expression $ \vert u_1^2-u_2^2 \mp \frac{1}{P^2} \vert < \frac{\lambda}{2}$ by interchanging the variables $u_2$ and $u_1$. Therefore we need only restrict our attention to the case $\vert u_1^2 - u_2^2 - \frac{1}{P^2} \vert < \frac{\lambda}{2}.$
Noting that we may take $\lambda$ small enough that $$\frac{1}{P^2} - \frac{\lambda}{2} > 0,$$ we see that the region of integration is the space inside the box $[0,1]^2$ between the two hyperbolae $u_1^2-u_2^2 = \frac{1}{P^2} + \frac{\lambda}{2}$ and $u_1^2-u_2^2 = \frac{1}{P^2} - \frac{\lambda}{2}$. The major contribution occurs before the lower hyperbola meets the line $u_1=1$. The rest of the integral will contribute $o(\lambda).$ With this in mind write $$I = I_1 + I_2,$$
where
\begin{align*}
I_1 &= \int \limits_0^{\sqrt{1- 1/P^2 - \lambda/2}} u_2 \log u_2 \int \limits_{\sqrt{u_2^2+1/P^2 - \lambda/2}}^{\sqrt{u_2^2+ 1/P^2 + \lambda/2}} u_1 \log u_1 \mathrm{d}u_1 \mathrm{d}u_2 \\
I_2 &= \int \limits_{\sqrt{1- 1/P^2 - \lambda/2}}^{\sqrt{1- 1/P^2 + \lambda/2}} u_2 \log u_2 \int \limits_{\sqrt{u_2^2 +1/P^2 - \lambda/2}}^1 u_1 \log u_1 \mathrm{d} u_1 \mathrm{d}u_2.
\end{align*}
The result of the inner integral of $I_1$ is 
\begin{equation}\label{I1}\frac{1}{4}\left[
  \left(c + \frac{\lambda}{2}\right)\left(\log \left(c + \frac{\lambda}{2}\right)-1\right)- \left(c - \frac{\lambda}{2}\right)\left(\log \left(c - \frac{\lambda}{2}\right)-1 \right)\right],
\end{equation} where $$c = u_2^2 + \frac{1}{P^2}.$$

\noindent
By considering the Taylor series of $\log(c\pm\frac{\lambda}{2})$, we write this expression as
$$I_1 =\frac{\lambda}{4}  \int \limits_0^{\sqrt{1- 1/P^2 -\lambda/2}} u_2 \log u_2 \log\left(u_2^2+\frac{1}{P^2} \right) \mathrm{d}u_2 + O\left(\lambda^3\right).$$
We now split the range of this integral according to whether $\vert u^2P^2\vert <1$ or $\vert u_2^2P^2 \vert > 1$. 
In the first of these ranges, we have
\begin{align*}
\int_0^{1/P}u_2 \log u_2 \log(u_2^2 + \frac{1}{P^2}) \mathrm{d}u_2 &=\int_0^{1/P}u_2 \log u_2 \left[ -2\log P+ O(u_2^2P^2)\right] \mathrm{d}u_2 \\
&\ll \frac{\log^2 P}{P^2}.
\end{align*}
Meanwhile, in the second range,
\begin{align*}
\int_{1/P}^{\sqrt{1-1/P^2 - \lambda/2}}u_2& \log u_2 \log(u_2^2 + \frac{1}{P^2}) \mathrm{d}u_2\\ &= \int_{1/P}^{\sqrt{1-1/P^2 - \lambda/2}}u_2 \log u_2 \left[2\log u_2 + O(P^{-2}u_2^{-2})\right] \mathrm{d}u_2\\
&=2\int_{1/P}^{\sqrt{1-1/P^2 - \lambda/2}}u_2\log^2u_2 \mathrm{d}u_2 +O(\frac{\log P}{P^2})\\
&= \frac{1}{2} + O(\frac{\log^2 P}{P^2}).
\end{align*}

We now turn to estimating $I_2$ by bounding above trivially then using the binomial theorem to see
\begin{align*}
I_2 & \leq \left(\sqrt{1-\frac{1}{P^2} + \frac{\lambda}{2}}-\sqrt{1-\frac{1}{P^2} - \frac{\lambda}{2}}\right) \left( 1- \sqrt{1-\lambda}\right)e^{-2}
\\&= e^{-2} \left(\frac{1}{2} \lambda+ \frac{1}{4} \frac{\lambda}{P^2} + \dots\right)(1- \sqrt{1-\lambda}),
\end{align*}
which clearly tends to $0$ after dividing by $\lambda$ and taking the limit $\lambda \rightarrow 0$.
\end{proof}

\begin{lemma}\label{2mu}
Let $(\beta, \gamma, \delta) \in L^{(i)}$.
For any choice of $\bm{\xi}\in T_{\beta,\gamma, \delta}^{\pm}$, we have
$$\mu_2^{\pm}(\bm{\xi}) = 2^{-12}.$$
\end{lemma}

\begin{proof}
Denote $$S(t) := \{\x \in \left(\mathbb{Z}/2^t \mathbb{Z}\right)^4 : F(\x) \equiv \pm 1 \text{ mod } 2^t \text{  and  } \x \equiv \bm{\xi}\text{ mod } 16\},$$
for $t\geq4$ so that $$\mu_2(P) = \lim_{t \rightarrow \infty} 2^{-3t} \#S(t).$$

\noindent
We'll use Hensel's lemma to relate this to $S(4)=\{\bm{\xi}\}$. In order to do this we first, note that 
\begin{align*}
 v_2(\nabla F(\x)) &= \min\{\delta+v_2(x_4),\beta + \gamma + v_2(x_3),\beta + \gamma + v_2(x_2),\delta+v_2(x_1)\}.
\end{align*}
If $\x \in S(t)$ for any $t \geq 4$ then $x_i \equiv \xi_i \Mod 16$ and therefore $v_2(x_i) =0$ for $i=1,\dots,4.$ Moreover either  $\delta=0$ or $\beta+\gamma =0$, since $(\beta, \gamma, \delta) \in L$. Hence $$  v_2(\nabla F(\x))  =0.$$
Thus by Hensel's lemma, we have
$$\#S(t) = 2^3 \#S(t-1) = \ldots =2^{3(t-4)}\#S(4) = 2^{3(t-4)},$$ which completes the proof.
\end{proof}

Combining these results with expression (\ref{total}) and writing $$T_{\beta,\gamma, \delta} = T_{\beta, \gamma, \delta}^+\sqcup T_{\beta, \gamma, \delta}^-,$$  we see that
\begin{equation} \label{NP} N(P) = \frac{P^2 }{2^8\pi^2}\sum_{i \in \{1,2\}}\sum_{(\beta, \gamma, \delta) \in L^{(i)}} \frac{\#T_{\beta, \gamma, \delta}}{2^{\beta + \gamma + \delta}} + O(P^{2 - \theta}).\end{equation}
This inner sum is easily verified to be $2^{13}$ completing the proof of Theorem \ref{Stot}.\\

The other counting functions can be evaluated similarly, using Lemma \ref{counting} again although with a slightly different setup. We will make the same change of variables as before although this time the set of allowable signs and congruence classes will be informed by Section 3. In particular, \eqref{def} clearly has no real solutions if ${\sigma(b,c,d) = (+,-,-)}$, so we exclude this possibility.
Then, analogously to above our local solution counting function can be expressed as
\begin{align*}
N_{\text{loc}}(P) &=  \frac{1}{4}\sum_{\bm{\epsilon}\in S \setminus\{(+,-,-)\}} \sum_{i \in \{1,2\}}\sum_{(\beta, \gamma, \delta) \in L^{(i)}} \sum_{\bm{\xi}\in H_{\beta, \gamma, \delta}^{\pm}} \left(K^+ + K^- \right). 
\end{align*}
The $K^{\pm}$ appearing in this expression is exactly the same as in \eqref{total} so we may use Lemma \ref{counting} to write this as

\begin{equation}\label{loc} N_{\text{loc}}(P) = \frac{P^2 }{2^{10}\pi^2} \sum_{\bm\epsilon \in S \setminus\{(+,-,-)\}}\sum_{i \in \{1,2\}} \sum_{(\beta, \gamma, \delta) \in L^{(i)}} \frac{\#H_{\beta, \gamma, \delta}}{2^{\beta + \gamma + \delta}} + O(P^{2 - \theta}).\end{equation}

The conditions describing $H_{\beta, \gamma, \delta}$ are spelled out explicitly by Lemmas 4.8-4.15 of \cite{BandB}. Crucially, the size of $H_{\beta, \gamma, \delta}$ doesn't depend on ${\bm \epsilon}$, so one can compute the sum
\begin{equation}\label{locdef} \tau_{loc, 2} :=2^{-13} \sum_{i \in \{1,2\}}\sum_{(\beta, \gamma, \delta) \in L^{(i)}} \frac{\#H_{\beta, \gamma, \delta}}{2^{\beta + \gamma + \delta}}>0\end{equation}
to obtain the expression \eqref{N_loc}.

The evaluation of $N_{\text{Br}}(P)$ will follow similar lines except this time counting over different congruence classes mod 16, again described explicitly in \cite{BandB} and with a smaller set of allowable signs. Recall from Section 3 that whenever $\sigma(b,c,d)=(+,-,+)$ the surface $X_{a,b,c,d}$ always satisfies the Hasse principle. Therefore analogously to \eqref{loc}, we have
\begin{equation}\label{NBr} N_{\text{Br}}(P) = \frac{P^2}{2^{10}\pi^2} \sum_{\bm\epsilon \in S \setminus\{(+,-,-),(-,+,-)\} }\sum_{i \in \{1,2\}} \sum_{(\beta, \gamma, \delta) \in L^{(i)}} \frac{\#\widetilde{H}_{\beta, \gamma, \delta}}{2^{\beta + \gamma + \delta}} + O(P^{2 - \theta}).\end{equation}
where
\begin{equation}\label{Brdef} 0<2^{-13} \sum_{i \in \{1,2\}} \sum_{(\beta, \gamma, \delta) \in L^{(i)}} \frac{\#\widetilde{H}_{\beta, \gamma, \delta}}{2^{\beta + \gamma + \delta}} =: \sigma_{loc, 2} \leq \tau_{loc,2}.\end{equation}
This completes the proof of Theorem \ref{theorem}.

All that remains is to give a value to the 2-adic densities $\tau_{loc,2}, \sigma_{loc,2}$ defined above. In \cite{BandB}, the congruences conditions controlling local and global solubility were worked out explicitly. We choose not to reproduce these here as they are rather lengthy and numerous, and do not illuminate the reader. Instead, we simply record the following table of values for $\# H_{\beta, \gamma, \delta}$ and $\#\widetilde{H}_{\beta,\gamma, \delta}$ which was obtained by using MATLAB to enumerate the residue classes in $\left((\ZZ/16\ZZ)^{\times}\right)^4$ which satisfy the conditions outlined in \cite{BandB}. See Table 1 for the outcomes.

\begin{table}
\caption{Computing the size of $H_{\beta,\gamma,\delta}$ and $\widetilde{H}_{\beta,\gamma,\delta}$}
\begin{tabular}{| c | c | c | c | c |}
\hline
$\bm{\beta}$ & $\bm{\gamma}$ & $\bm{\delta}$ & $\bm{\#H_{\beta, \gamma, \delta}}$ & $\bm{\#\widetilde{H}_{\beta,\gamma,\delta}}$ \\
\hline
0 & 0 & 1 & 1024 & 416 \\
\hline
0 & 0 & $\geq$2 even &960 &512 \\
\hline
0&0&$\geq$3 odd &1024&544\\
\hline
0&1&0&1024&416\\
\hline
0&$\geq$2 even&0&960&448\\
\hline
0&$\geq$3 odd&0&1024&416\\
\hline
1&0&0&1024&192\\
\hline
1&1&0&1024&128\\
\hline
1&$\geq$2 even&0&1024&320\\
\hline
1&$\geq$3 odd&0&1024&256\\
\hline
2&0&0&992&480\\
\hline
2&1&0&1024&576\\
\hline
2&$\geq$2&0&768&320\\
\hline
3&0&0&1024&320\\
\hline
3&1&0&1024&384\\
\hline
3&$\geq$2&0&768&128\\
\hline
$\geq$4 even&0&0&992&480\\
\hline
$\geq$4 even&1&0&1024&576\\
\hline
$\geq$4 even&$\geq$2&0&768&320\\
\hline
$\geq$4 odd&0&0&1024&320\\
\hline
$\geq$4 odd&1&0&1024&384\\
\hline
$\geq$4 odd&$\geq$2&0&768&128\\
\hline
\end{tabular}
\end{table}

With these values, it is simple to conclude from \eqref{locdef} that $$\tau_{loc,2} = \frac{17856}{3\times 2^{13}}.$$ Likewise, we compute $$\sigma_{loc,2} = \frac{2112}{2^{13}}.$$


\begin{thebibliography}{9}
\bibitem{BandB}
R. De La Bret\`{e}che and T. D. Browning, Density of Ch\^{a}telet surfaces failing the Hasse principle, {\it Proc. London Math. Soc.} {\bf 108} (2014) 1036-1078.
\bibitem{BandG}
T. D. Browning and A. Gorodnik, Power-free values of polynomials on symmetric varieties, {\it Proc. Lond. Math. Soc.} {\bf 114} (2017) 1044-1080.
\bibitem{CTCS}
J.-L. Colliot-Th{\'e}l{\`e}ne, D. Coray and J.-J. Sansuc, Descent and the Hasse principle for certain rational varieties, {\it J. reine angew. Math.} {\bf 320} (1980) 150-191.
\bibitem{CTSSD}
J.-L. Colliot-Th{\'e}l{\`e}ne,  J.-J. Sansuc and P. Swinnerton-Dyer, Intersections of two quadrics and Ch{\^a}telet surfaces, I, {\it J. reine angew. Math.} {\bf 373} (1987) 37-107; II {\it J. reine angew. Math.} 374 (1987) 72-168.
\bibitem{Isk}
V. A. Iskovskikh, A counterexample to the Hasse principle for systems of two quadratic forms in five variables, {\it Mat. Zametki} {\bf 10} (1971) 253-257.



\end{thebibliography}
\end{document}